\documentclass[11pt,reqno]{amsart}

\usepackage{amsmath}
\usepackage{amssymb}
\usepackage[T1]{fontenc}
\usepackage{mathrsfs}
\usepackage{latexsym}
\usepackage{color}
\usepackage[dvipsnames]{xcolor}
\usepackage{tikz}
\usetikzlibrary{patterns}
\numberwithin{equation}{section}

\newcommand{\R}{\mathbb{R}}

\newcommand{\Z}{\mathbb{Z}}
\newcommand{\N}{\mathbb{N}}
\newcommand{\T}{\mathbb{T}}
\newcommand{\F}{\mathbb{F}}

\newcommand{\diff}{\mathrm{d}}
\newcommand{\re}{\mathrm{Re\,}}

\DeclareMathOperator{\dom}{dom}
\DeclareMathOperator{\supp}{supp}

\normalsize

\newtheorem{theorem}{Theorem}[section]
\newtheorem*{main theorem}{Main Theorem}
\newtheorem{lemma}[theorem]{Lemma}
\newtheorem{definition}[theorem]{Definition}
\newtheorem{remark}[theorem]{Remark}
\newtheorem{assumption}{Assumption}

\definecolor{luh-dark-blue}{rgb}{0.0, 0.313, 0.608}
\definecolor{lred}{rgb}{1.0,0.5,0.5}

\title[Well-posedness in dispersive equations]{Classical well-posedness in dispersive equations with nonlinearities of mild regularity, and a composition theorem in Besov spaces}
\author{Mats Ehrnstr\"{o}m}

\address{Department of Mathematical Sciences, NTNU Norwegian University of Science and Technology, 7491 Trondheim, Norway}
\email{mats.ehrnstrom@math.ntnu.no}

\author{Long Pei}

\address{Department of Mathematics, KTH Royal Institute of Technology, 10044 Stockholm, Sweden
\newline
Department of Mathematical Sciences, NTNU Norwegian University of Science and Technology, 7491 Trondheim, Norway
}
\email{longp@kth.se}

\thanks{\color{black} Both authors acknowledge the support by grant no. 231668 by the Research Council of Norway. M.E. additionally acknowledges the support by grant no. 250070 by the same source.}



\usepackage[colorlinks=true]{hyperref}
\hypersetup{urlcolor=blue, citecolor=red, linkcolor=blue}

\begin{document}

\begin{abstract}
For both localized and periodic initial data, we prove local existence in classical energy space $H^{s}$, $s>\frac{3}{2}$, for a class of dispersive equations \(u_t + (n(u))_x + L u_x = 0\) with nonlinearities of mild regularity. Our results are valid for symmetric Fourier multiplier operators $L$ whose symbol is of temperate growth, and $n(\cdot)$ in the local Sobolev space $H^{s+2}_{\mathrm{loc}}(\R)$. In particular, the results include non-smooth and exponentially growing nonlinearities. Our proof is based on a combination of semi-group methods and a new composition result for Besov spaces. In particular, we extend a previous result for Nemytskii operators on Besov spaces on \(\R\) to the periodic setting by using the difference--derivative characterisation of Besov spaces.
\end{abstract}

\maketitle


\section{Introduction}\label{sec:introduction}
We consider nonlinear dispersive equations of the form
\begin{equation}\label{eq:main}
u_{t}+(n(u))_{x}+\mathit{L}u_{x}=0,
\end{equation}
where $n$ denotes the nonlinearity and the linear (dispersive) operator $\mathit{L}$ is defined as a Fourier multiplier operator by
\begin{equation}\label{eq:dis-def}
\mathcal{F}(\mathit{L}f)(\xi)=m(\xi)\mathcal{F}(f)(\xi),
\end{equation}
for some real and measurable function $m$. A large class of equations, including the  Korteweg--de Vries (KdV) \cite{colliander2003sharp} and Benjamin--Ono (BO) equations \cite{ionescu2007global},  are covered by \eqref{eq:main}. But our main inspiration comes from \cite{EGW11}, in which the existence of solitary waves was established  for a class of nonlocal equations of Whitham type \eqref{eq:main}, and energetic stability of solutions was obtained based on an a priori well-posedness assumption. The Whitham equation itself corresponds to the nonlinearity 
$u^{2}$ and nonlocal dispersive operator $L $ with Fourier symbol $\sqrt{\tanh(\xi)/\xi}$,
and is one of several equations of the form \eqref{eq:main} arising in the theory of water waves \cite{lannes2013water}.\\

Equations of the form \eqref{eq:main} are in general not completely integrable, and one has to apply contraction or energy methods to obtain estimates for both the linear and nonlinear terms to prove existence of solutions.  Our linearity will be skew-adjoint (the symbol $m(\xi)$ will be real and even), but the only additional assumption is that \(m\) is of moderate growth, that is, 
\[
|m(\xi)| \lesssim (1+|\xi|)^l,
\]
for some \(l \in \R\) and all \(\xi \in \R\), which is to guarantee that the domain of $L$ is dense in $L_{2}$. Note that this class of symbols is very large, covering  both homogeneous and inhomogeneous symbols. For the same reason scaling argument cannot be applied, and the difficulties increase in determining the critical energy spaces for \eqref{eq:main} and in acquiring the decay of solutions over time needed for the global existence of solutions. Moreover, if  $l$ above is sufficiently negative then the dispersion of $L$ is very weak,  and global well-posedness of \eqref{eq:main} fails in classical energy spaces. In fact, the Whitham equation as a typical representative of \eqref{eq:main} is  locally well-posed in Sobolev spaces $H^{s}$, $s>3/2$, with localized or periodic initial data \cite{raey,KLPS17}, but exhibits finite-time blow-up (wave-breaking) for some sufficiently smooth initial data \cite{0802.35002,MR1668586,vera2015hur} so that  global well-posed in $H^{s}(\R)$, $s>3/2$, is not possible. This kind of break-up phenomena cannot be observed in equations with strong dispersion like KdV and BO, 
{\color{black} which are globally well-posed in $H^{s}(\R)$ for all \(s \geq 1\) (see \cite{colliander2003sharp} and \cite{tao2004global}, respectively).}\\

Our main concern, however, is the nonlinearity $n(\cdot)$. The most well-studied nonlinearities are pure-power type nonlinearities $u^{p}$, $p \in \Z_+$. For a fixed $p$, such nonlinearities have moderate growth rate far away from the origin and it is possible to adjust $p$ so that solutions will exist globally provided that dispersion is not too weak \cite{cascaval2004local}. Otherwise, if we fix the dispersion but allow the nonlinearity to grow fast enough, the solutions may blow up within finite time  \cite{bona1995conservative}. Another important feature of pure power nonlinearities is that they define smooth, regularity-preserving maps on Sobolev spaces $H^{s}$, $s>1/2$, in one dimension so that it is easy to obtain a contraction mapping from the solution operator based on either Duhamel's principle or classical energy estimates. These features, however, fail to hold for general nonlinearities. The nonlinearity in our consideration shall belong to the local Sobolev space $H^{s+2}_{\mathrm{loc}}(\R)$,  $s>3/2$. For such nonlinearities and dispersive operators $L$ as mentioned above, we establish local well-posedness for data in $H^{s}$, $s>3/2$. Due to the locality of the space $H^{s}_{\mathrm{loc}}(\R)$ our result shows that dispersive equations of type \eqref{eq:main} are locally well-posed in high-regularity spaces even for nonlinearities $n$ of arbitrarily fast growth, and for dispersive operators $L$ with arbitrarily weak or strong dispersion (as long as the dispersion is not extreme in the sense that \(L \not\in \mathcal{S}^\prime(\R)\), the space of distributions on Schwartz space). It is a point in our investigation that the nonlinearity $n\in H^{s+2}_{\mathrm{loc}}(\R)$ is not necessarily smooth, and covers standard nonlinearities of the form $u^p$ and $|u|^{p-1}u$ as well as others. Naturally, the regularity of the nonlinearity will affect the energy space, as can be seen in Theorem~\ref{main theorem}.\\

The local space $H^{s+2}_{\mathrm{loc}}(\R)$ first comes into play in obtaining the required commutator estimate for an operator-involving nonlinearity, needed in the semi-group theory. The key for such a commutator estimate in our case are the mapping properties of the nonlinearity $n$ over the energy spaces in consideration, which is guaranteed by a composition theorem  \cite{bourdaud2014composition} when the initial data is in Sobolev spaces on the line.  Namely, a composition operator $T_{f}$ is said to  \emph{act on} a function space $X$ if 
 \[
 T_{f}(g) =f \circ g \subset X,
 \]
 for all $g\in X$. The composition theorem in \cite{bourdaud2014composition} holds that a function $f$ acts on Besov spaces $B_{pq}^{s}(\R)$, $1<p<\infty$, $0<q\leq \infty$, $s>1+1/p$,   if, \emph{and only if}, $f\in B^{s}_{pq,\mathrm{loc}}(\R)$ and $f(0)=0$. That composition theorem, when applied to a non-smooth nonlinearity $n$, guarantees a contraction mapping in the desired solution spaces. We mention here that the Cauchy problem \eqref{eq:main} with a non-smooth nonlinearity is also considered in \cite{cascaval2004local}, in which the author establishes global well-posedness in $H^{\sigma}(\R)$, $\sigma=\max\{\alpha, 3/2+\epsilon\}$ with  $|\alpha|\geq 1$, $\epsilon>0$,  by nonlinear semigroup theory when the dispersion is not too weak ($m(\xi)=|\xi|^{\alpha}$). The nonlinearity in their consideration belongs to the H{\"o}lder space $C^{\alpha+1}(\R)$. Similar well-posedness results in $H^{s}$, $s>1/2$, for generalized KdV equations can be found in \cite{staffilani1997generalized}.\\

In the periodic setting a composition theorem like the one in \cite{bourdaud2014composition} is not available in the literature. However, we prove in Theorem~\ref{periodic composition theorem} that \(f\) acts on the periodic Besov space \(B^{s}_{pq}(\T)\), $1<p<\infty$, $1< q\leq \infty$ and $s>1+\frac{1}{p}$, if, \emph{and only if}, $f\in B^{s}_{pq,\mathrm{loc}}(\R)$, where $\T$ denotes the one-dimensional torus (the requirement that $f(0)=0$ appearing in \cite{bourdaud2014composition} is not necessary due to the periodicity, else the results are comparable). 
Our proof relies on the composition theorem for non-periodic Besov spaces and what can be called a localizing property of periodic Besov spaces (see Lemma \ref{per nonper control}): smooth but compactly supported extensions of periodic functions from a n-dimensional torus $\T^{n}$ to the whole space $\R^{n}$ will be controlled by the latter periodic function norms. A feature of our proof is that we work directly with the difference--derivative characterization of Besov spaces, not the Littlewood--Paley decompositions (used for example in \cite{bourdaud2014composition}). In this case, the main difficulty arises at integer regularity $s \in \N$, where the estimates are particularly cumbersome. It could be mentioned that the composition theorem on \(\R\) mentioned earlier has only been proved in one dimension, even though it is conjectured to hold in higher dimensions too (cf.  \cite{bourdaud2014composition}). Finally, we mention that a related paper for the periodic setting is \cite{hu2013discrete}, in which KdV-type equations  are considered, and which is based on the work \cite{bourgain1995cauchy} of Bourgain.\\

The outline is as follows. In Section \ref{assumptions and main results} we state our assumptions and the main result. In Section \ref{sec:preliminiaries} we reformulate \eqref{eq:main} in terms of the semigroup $e^{-tL\partial_{x}}$ generated by $L$.  Following \cite{0934.35153} and \cite{MR535697} we then introduce Kato's method from \cite{MR0407477}.  The same section concludes with the composition theorem for operators on Besov spaces on the line. In Section \ref{sec:realcase} we give a detailed analysis of the generator of semigroup $e^{-tL\partial_{x}}$  and the proof of our local well-posedness result on the line. Attention needs to be paid in the case of a general dispersive symbol and a nonlinearity of mild regularity. In particular, one finds a proper domain and an equivalent definition for the dispersive operator $L\partial_x$ on which it is closed and generates a continuous semigroup. Finally, in Section \ref{the case of periodic initial data} we prove the above-mentioned composition theorem for Besov spaces on the torus $\T$ (see Theorem~\ref{periodic composition theorem}), this being the key for the well-posedness for periodic initial data. The proof of Theorem~\ref{periodic composition theorem} is based on the localizing property detailed in Lemma \ref{per nonper control}. With all this in place, local well-posedness on $H^{s}(\T)$ is straightforward, as it can be acquired analogously to the non-periodic case, and we only state the idea and point out some key points.

\section{Assumptions and main results}\label{assumptions and main results}
In \eqref{eq:main}, the Fourier transform $\mathcal{F}(f)$ of a function $f$ is defined by the formula 
\begin{equation*}
\mathcal{F}(f)(\xi)=\frac{1}{\sqrt{2\pi}}\int_{\R}e^{-i\xi x}f(x)\, \diff x,
\end{equation*}
extended by duality from the Schwartz space of rapidly decaying smooth functions $\mathcal{S}(\R)$ to the space of tempered distributions $\mathcal{S}^{'}(\R)$. The notation $\hat{f}$ will  be used interchangeably with $\mathcal{F}(f)$. For all \(s\in \R\), we denote by $H^{s}(\R)$ the Sobolev space of tempered distributions whose Fourier transform satisfies 
\begin{equation*}
\int_{\R}(1+|\xi|^{2})^{s}|\mathcal{F}(f)(\xi)|^{2}d\xi <\infty,
\end{equation*}
and equip it with the induced inner product (the exact normalization of the norm will not be important to us). Sobolev spaces are naturally extended to the Besov spaces $B^{s}_{pq}(\R^{n})$, \(s >0\),  \(0 < p,q \leq \infty\), however, we will postpone the exact definition for Besov spaces and introduce them in the periodic setting, since the well-known {\color{black} identity} $H^{s}(\R^{n})=B^{s}_{22}(\R^{n})$, $s>0$, is the only property of Besov spaces that we will use for the well-posedness for localized initial data. Finally, we use $C_{c}^{\infty}$ to denote the $C^{\infty}$ {\color{black} (smooth)} functions with compact support. Note that although the solutions of interest in this paper are all real-valued, and although the operator $L$ defined in \eqref{eq:main} maps real data to real data, the Fourier transform is naturally defined in complex-valued function spaces. Hence, the function spaces used in this investigation should in general be understood as consisting of complex-valued functions. For convenience, we shall sometimes omit the domain in the notation for function spaces, and {\color{black} we use the notation $A\lesssim B$ if there exists a positive contant $c$ such that $A \leq c B$.}\\

Our assumptions for \eqref{eq:main} are stated in Assumption~\ref{main assum}. Note that we have no regularity assumptions for \(m\) except for it being measurable, but the below conditions will guarantee that \(m \in L_{\infty,\mathrm{loc}}\). The growth condition on \(m\) is to guarantee the density of the domain of the linear operator in $L_{2}(\R)$. As what concerns the nonlinearity, the assumption on it is completely local. In particular, \(n(x)\) may grow arbitrarily fast as \(|x| \to \infty\).

\begin{assumption}\label{main assum}
$ $\\[-12pt]
\begin{itemize}
\item[(A1)] The operator $L$ is a symmetric Fourier multiplier, that is,
\begin{equation*}
\mathcal{F}(Lf)(\xi)=m(\xi)\hat{f}(\xi)
\end{equation*} 
for some real and even measurable function $m:\mathbb{R}\to \mathbb{R}$.\\

\item[(A2)] The symbol $m$ is temperate, that is, 
\[
|m(\xi)| \lesssim (1+|\xi|)^l,
\]
for some constant \(l\in\R\) and for all \(\xi \in \R\).\\

\item[(A3)] There exists \(s > \frac{3}{2}\) such that the nonlinearity \(n\) belongs to \(H^{s+2}_{\mathrm{loc}}(\R)\).
\end{itemize}
\end{assumption}
%

{\color{black} Denote by  $\mathbb{T}$ the one-dimensional torus of circumference \(2\pi\).} We now state our main theorem, yielding local existence for a fairly large class of equations. Note in the following that when \(l > s -1\),  the Sobolev space  $H^{s-l-1}$ becomes strictly larger than \(L_{2}\), although our interest mainly arises from the case \(l < 0\).

\begin{theorem}\label{main theorem}
Let $\F\in \{\R,\T\}$. Under Assumption~\ref{main assum} and given $u_{0}\in H^{s}(\F)$ for $s>\frac{3}{2}$, there is a maximal $T>0$ and a unique solution $u$ to \eqref{eq:main} in $C([0,T); H^{s}(\F))\cap C^{1}([0,T);H^{s-\max\{1, {\color{black} l+1}\}}(\F))$. The map $u_{0}\mapsto u(\cdot, u_{0})$ is continuous between the above function spaces.
 \end{theorem}

{\color{black} 
\begin{remark}
Note that the H{\"o}lder spaces \(C^{k,\alpha} (\mathbb{R})\) are continuously embedded into the local Sobolev spaces \(H_{\mathrm{loc}}^{s} (\mathbb{R})\) for \(k = s\) when \(\alpha = 0\) and for \(k + \alpha  > s\) when \(\alpha \in (0,1)\). Hence,  H{\"o}lder-continuous functions are concrete examples of the nonlinearities considered in this paper. 
\end{remark}

As described in the introduction, our proof of Theorem \ref{main theorem} relies on a combination of Kato's classical energy method, and recent composition theorems for Besov spaces (see Theorems~ \ref{composition theorem} and~\ref{periodic composition theorem}).}

\section{General Preliminaries}\label{sec:preliminiaries}
Consider the transformation 
\begin{equation}\label{pre-transf}
u(t)=e^{-tL\partial_x}v(t),
\end{equation}
where a detailed analysis of the semigroup $e^{-tL\partial_x}$, $t\geq0$, is given in Section \ref{oper L-partial x}.
Substitution of \eqref{pre-transf} into \eqref{eq:main} yields the quasilinear equation
\begin{equation}\label{transed-eq}
\frac{dv}{\, \diff t}+A(t,v)v=0
\end{equation}
for the new unknown $v$, where 
\begin{equation}\label{eq:A}
A(t,y)=e^{tL\partial_x}  [n'(e^{-tL\partial_x}y) \partial_x ]  e^{-tL\partial_x},
\end{equation}
and $n'(e^{-tL\partial_x}y)$ acts by pointwise multiplication. Given a function $y\in H^{\frac{3}{2}+}$ the operator $u \mapsto n'(e^{-tL\partial_x}y)u$ is a bounded linear operator $L_{2} \to L_{2}$. Therefore, $A(t,y)$ is a well-defined bounded linear operator \(H^1 \to L_{2}\). For the purpose of applying Theorem~\ref{Kato's theorem} below, we fix an arbitrary value of \(s > \frac{3}{2}\) and consider spaces
\begin{equation}\label{eq:XY}
X=L_{2}(\mathbb{R}) \qquad\text{ and }\qquad Y=H^{s}(\mathbb{R}),
\end{equation}
between which \(\Lambda^s = (1 - \partial_{x}^2)^{\frac{s}{2}}\) defines a topological isomorphism (isometry, under the appropriate norms).
Let furthermore $B_R$ be the open ball of radius \(R\) in $H^{s}$, for an arbitrary {\color{black} but fixed} radius \(R > 0\). \\

To state Theorem~\ref{Kato's theorem}, let  \(T\) be an operator on a Hilbert space \(H\). We denote the space of all bounded linear operators on \(H\) by \(\mathcal{B}(H)\). Following \cite{Kato95}, we call an operator \(T\) on a Hilbert space $H$ \emph{accretive} if 
\[
\re \langle Tv, v \rangle_{H}\geq 0 
\]
holds for all \(v\in \dom(T)\), and \emph{quasi-accretive} if $T+\alpha$ is accretive for some $\alpha\in \R$. If  $(T+\lambda)^{-1}\in \mathcal{B}(H)$ with 
\[
\|(T+\lambda)^{-1}\|_{\mathcal{B}(H)}\leq (\re \lambda)^{-1} 
\]
for \(\re\lambda>0\), then \(T\) will be called \emph{m-accretive}, and if $T+\alpha$ is m-accretive for some scalar $\alpha\in \mathbb{R}$ it will be called \emph{quasi-m-accretive}.

\begin{theorem}\cite{MR0407477}\label{Kato's theorem} 
Let $X,Y, B_R$ and $\Lambda^s$ be {\color{black} as above}. Consider the quasi-linear Cauchy problem
\begin{equation}\label{general form of equation}
\frac{d u}{d t}+A(u)u= 0,\quad t\geq 0, \quad u(0)=u_{0},
\end{equation}
and assume that\\[-6pt]
\begin{itemize}
\item[(i)] $A(y)\in \mathcal{B}(Y,X)$ for $y\in B_R$, with 
\begin{equation*}
\|(A(y)-A(z))w\|_{X}\lesssim \|y-z\|_{X}\|w\|_{Y},\quad y,z,w \in B_R,
\end{equation*}
and $A(y)$ {\color{black} uniformly} quasi-m-accretive on $B_R$.\\[-6pt]

\item[(ii)] $\Lambda^s A(y)\Lambda^{-s}=A(y)+B(y)$, where $B(y)\in \mathcal{B}(X)$ is uniformly bounded on \(B_R\), {\color{black} and}
\begin{equation*}
\|(B(y)-B(z))w\|_{X} \lesssim \|y-z\|_{Y}\|w\|_{X},\quad y,z\in B_R, \quad w\in X.
\end{equation*}
\end{itemize}
Then, for any given ${\color{black} u_{0}} \in Y$, there is a maximal $T>0$ depending only on $\|u_{0}\|_{Y}$ and a unique solution $u$ to \eqref{general form of equation}  such that 
\begin{equation*}
u=u(\cdot , u_{0})\in C([0,T);Y)\cap C^{1}([0,T);X),
\end{equation*}
where the map $u_{0}\mapsto u(\cdot , u_{0})$ is continuous $Y \to C([0,T);Y)\cap C^{1}([0,T);X)$.
\end{theorem}

The continuity of the operators $A$ and $B$ in Theorem \ref{Kato's theorem} can be reduced to a commutator estimate. This is where our composition theorem will play a role in order to control the nonlinearity $n$. To this aim, we introduce the concept of \emph{action} of a composition operator.

\begin{definition}[{\color{black} Action} property]
For a function $f$, let $T_{f}$ denote the composition operator \(\color{black} g\mapsto f \circ g\). The operator $T_{f}$ is said to \emph{act on} a function space \(W\) if for any $g\in W$ one has $T_{f}(g)\in W$, that is,
\begin{equation*}
T_{f} W\subset W.
\end{equation*}
\end{definition}

For some Besov spaces over \(\R\), {\color{black} the set of acting functions has been} completely characterised in terms of a local space, as described in the following result. 

\begin{theorem}\cite{bourdaud2014composition} \label{composition theorem}
Let $1<p<\infty$, $0< q\leq \infty$ and $s>1+(1/p)$. For a Borel measurable function $f:\mathbb{R}\to\mathbb{R}$, the composition operator $T_{f}$ acts on $B^{s}_{pq}(\R)$ exactly when $f(0)=0$ and $f\in B^{s}_{pq,\mathrm{loc}}(\mathbb{R})$. In that case, $T_{f}$ is {\color{black} bounded}. 
\end{theorem}

It is obvious that for spaces of functions with decay the condition \(f(0) = 0\) is necessary. It is clear, too, that \(f\) necessarily must have the same (local) regularity as prescribed by the space \(B^s_{pq}\), since otherwise a smoothened cut-off of the function \(x \mapsto x\) would be mapped out of the space by \(f\). What is less obvious is that these two properties are actually equivalent to the action property; and in the one-dimensional case in fact to boundedness of \(f\) on \(B^s_{pq}(\R)\), see \cite{bourdaud2014composition}. We record here also the following definition and consequence of Theorem~\ref{composition theorem}.

{\color{black}
\begin{remark}\label{norm control for compsosition} Let $\phi\in C_{c}^{\infty}(\R)$ be a smooth cut-off function such that $0 \leq \phi \leq 1$,
\begin{equation}
\phi(x)=1\quad \mathrm{for} \quad|x|\leq 1,
\end{equation}
and $\supp(\phi) \subset [-2,2]$. For $a>0$, we denote by 
\[
\varphi_{a}(x) =\phi \left(\frac{x}{a} \right)
\] 
the \(a^{-1}\)-dilation of $\phi$. Then the following inequality is a consequence of the proof of Theorem~\ref{composition theorem} in  \cite{bourdaud2014composition}. For \(a = \|g\|_{L_{\infty}}\), one has
\begin{equation}\label{composition bound}
\|f(g)\|_{B^{s}_{pq}(\R)}\lesssim \|(f\varphi_{a})'\|_{B^{s-1}_{pq}(\R)}
\|g\|_{B^{s}_{pq}(\R)}(1+\|g\|_{B^{s}_{pq}(\R)})^{s-1-\frac{1}{p}}.
\end{equation}
\end{remark}

We are now ready to move on to the existence result on \(\R\).}

\section{Localised initial data}\label{sec:realcase}

\subsection{The operator \(L\partial_x\)}\label{oper L-partial x}
{\color{black} Most of the material in this subsection is standard, and we present it in condensed form. Often, though, in the literature details are only given in the case when the assumptions on \(L\) are much more restrictive and \(n\) is a pure power nonlinearity.} For a classical paper on well-posedness of nonlinear dispersive equations, see, e.g., \cite{ABDELOUHAB1989360}. Here, we follow the route of \cite{0934.35153} and start by defining the domain of  $\partial_x L = L \partial_x$ in {\color{black} $ L_2(\R)$ (from now on only \(L_2\))} by
\begin{equation}\label{dom 1}
\dom(L\partial_x)=\{f\in L_{2}\colon L\partial_x f\in L_{2}\}.
\end{equation}

\begin{lemma}
Let $S$ be the set of all $f\in L_{2}$ for which there exists  $g\in L_{2}$ with
\begin{equation}\label{dom via integral}
\langle\cdot, g \rangle_{L_{2}} = -\langle L\partial_x \,\cdot,  f \rangle_{L_{2}}
\end{equation}
Then $\mathrm{dom}(L\partial_x )=S$ and \(L\partial_x  = [f \mapsto g]\). 
\end{lemma}

\begin{proof}
Since \(L\) is symmetric, for any $f\in \mathrm{dom}(L\partial_x)$ and any \(\phi \in C_c^\infty(\R)\) , we have
\[
\color{black}
\langle \phi, L\partial_xf \rangle_{L_{2}} = -\langle L\partial_x \phi, f \rangle_{L_{2}}.
\]
Thus $L\partial_xf\in L_{2}$ yields $\mathrm{dom}(L\partial_x)\subset S$. To see that $L\partial_xf = g$, note that
\begin{equation*}
\color{black}
\begin{split}
&\langle \phi, g  \rangle_{L_{2}} =- \langle  L\partial_x\phi, f \rangle_{L_{2}}=\ \langle L\partial_xf, \phi \rangle_{L_{2}},
\end{split}
\end{equation*}
for any \(f \in \dom (L\partial_x)\), {\color{black} by the skew-symmetry of \(L\partial_x\).} Thus, if {\color{black} we knew that} $S\subset \mathrm{dom}(L\partial_x)$, we could conclude $L\partial_x=[f \mapsto g]$. For $f$ in \(S\),  \(\langle L\partial_xf,\cdot\rangle=-\int_{\R}f(L\partial_x\cdot)\, \diff x\) is {\color{black} clearly} a well-defined distribution. In view of \eqref{dom via integral} we have \(\langle L\partial_xf-g,\phi\rangle=0\), and therefore
\begin{equation}
L\partial_xf=g  \quad \mathrm{in} \quad\mathcal{D}^\prime(\R).
\end{equation}
Since $g\in L_{2}$ we deduce that $L\partial_xf\in L_{2}$, which in turn implies  $f\in \mathrm{dom}(L\partial_x)$. This concludes the proof.
\end{proof}

We record the following properties of  $L\partial_x$. 
\begin{lemma}\label{prop LD}
$L\partial_x$ is densely defined on $L_{2}$, closed, and skew-adjoint on \(S\).
\end{lemma}
\begin{proof}
The denseness of \(\dom(L\partial_x)\) in \(L_{2}\) follows by (A2) (see Assumption~\ref{main assum}). Similarly, closedness in \(L_{2}\) follows from that \(L\) is a symmetric Fourier multiplier operator, cf. (A1). {\color{black} Now, let} $(L\partial_x)^{*}$ be the  \(L_{2}\)-adjoint of $L\partial_x$. Then, for any $g\in \mathrm{dom}((L\partial_x)^{*})$ and any $\phi {\color{black} \in} C_{c}^{\infty} \subset \dom(L\partial_x)$, we have
\[
\langle \phi, (L\partial_x)^{*}g\rangle_{L_{2}} =\langle L\partial_x\phi,  g \rangle_{L_{2}},
\]
which implies that $g\in \mathrm{dom}(-L\partial_x)$ and $(L\partial_x)^{*}\subset -L\partial_x$. To prove the inverse relation, we use that for any $f\in \mathrm{dom}(L\partial_x)$ there is a sequence $\{f_{n}\}\) of smooth functions such that 
\begin{equation}\label{appro seq ld}
f_{n}\stackrel{L_{2}}{\to} f \quad \mathrm{and} \quad (L\partial_x)f_{n} \stackrel{L_{2}}{\to} (L\partial_x)f\quad,
\end{equation}
as \(n \to \infty\). Here $f_{n} = f \ast \varrho_{n}$, where $\varrho_{n}(x) =n\varrho(nx)$ and $\varrho\in C_{c}^{\infty}$ is a mollifier satisfying $\varrho(x)\geq 0$ and $\int_{\mathbb{R}}\varrho\, \, \diff x=1$.  Since \(\varrho_n \in C_{c}^{\infty}\), we have  $f_{n}\in C^{\infty}\cap L_{2}$ and clearly $f_{n}\to f$ in $L_{2}$ as $n\to \infty$. Furthermore,
\begin{equation}\label{conv L1}
\begin{split}
\mathcal{F}(L\partial_xf_{n})=i\xi m(\xi)\mathcal{F}(f)\mathcal{F}(\varrho_{k})=\mathcal{F}(f)\mathcal{F}(L\partial_x\varrho_{n})=\mathcal{F}(f*(L\partial_x\varrho_{n}))
\end{split}
\end{equation}
and because $L\partial_x$ is a Fourier multiplier we also have that
\begin{equation}\label{conv L2}
f*(L\partial_x\varrho_{n}) = (L\partial_xf)*\varrho_{n}.
\end{equation}
By \eqref{conv L1} and  \eqref{conv L2}, one has
\begin{equation*}
\begin{split}
\|{\color{black} L\partial_x} f_{n}-L\partial_xf\|_{L_{2}} &=\|(L\partial_xf)*\varrho_{n}-L\partial_xf\|_{L_{2}}\to 0\quad \text{as}\quad n\to \infty,
\end{split}
\end{equation*}
which establishes the density of $C^{\infty}\cap L_{2}$ in the graph norm of \(L\partial_x\) on $\mathrm{dom}(L\partial_x)$.
For each $\color{black} g\in \mathrm{dom}(L\partial_x)$, we can find $\{g_{n}\}\subset C^{\infty}\cap L_{2}$. Then, for any $\color{black}  f\in \mathrm{dom}(L\partial_x)$ we have
\begin{equation*}
\begin{split}
 \langle L\partial_xf, g \rangle_{L_{2}} =\lim_{n\to \infty} \langle L\partial_xf, g_{n} \rangle_{L_{2}} =-\lim_{n\to \infty}\langle f, L\partial_xg_{n} \rangle_{L_{2}} =\langle f, -L\partial_x g \rangle_{L_{2}}.
\end{split}
\end{equation*}
Therefore, $-L\partial_x\subset (L\partial_x)^{*}$, and the operator $L\partial_x$ is skew-adjoint on \(S\).
\end{proof}

From lemma \ref{prop LD} and Stone's theorem one {\color{black} then obtains} the following standard result.\\

\begin{lemma}\label{fourier LD}
$L\partial_x$ generates a unitary group $\{e^{-tL\partial_x}\}$ on $H^{s}$, $s\geq 0$, where
\[
\mathcal{F}(e^{-tL\partial_x}f)=e^{i t\xi m(\xi)}\mathcal{F}(f), \qquad f\in L_{2}.
\]
\end{lemma}

\subsection{Properties of the operators \(A\) and \(B\)}
We now study the operator $A(t,y)$ for a fixed $y\in B_R \subset H^{s}$. All estimates to come are uniform with respect to such \(y\). To prove that the operator \(A(t,y)\) is quasi-m-accretive, we establish that both \(A(t,y)\) and its adjoint are quasi-accretive. It then follows by \cite[Corollary 4.4]{Pazy83} that \(A(t,y)\) is quasi-m-accretive, as proved in the following lemma.

\begin{lemma}\label{lemma:quasi-m}
For any fixed \(y \in B_R\), the operator $-A(t,y)$ is the generator of a $C_{0}$-semigroup on $X$, and \(A(t,y)\) is uniformly quasi-m-accretive on \(B_R\). In particular, for all \(y \in B_R\) one has the uniform estimate
\begin{equation}\label{quasi-acc-relation}
(A(t,y)w,w)_{X}\gtrsim -\|w\|_{X}^{2},
\end{equation}
for \(w\in C_{0}^{\infty}\).
\end{lemma}

\begin{proof}
With 
\begin{equation}
\mathrm{dom}(A(t,y))=\{u\in L_{2} \colon n'(e^{-tL\partial_x}y)e^{-tL\partial_x} u\in H^{1}\}, 
\end{equation}
$A(t,y)$ is densely defined in \(L_{2}\), and closed: Take $\{u_{n}\}\subset \mathrm{dom}(A(t,y))$ with $u_{n}\to u\in L_{2}$ and $A(t,y)u_{n}\to v\in L_{2}$. Abbreviate \(A = A(t,y)\). Then, for any $\phi\in C_{c}^{\infty}$,
\begin{equation*}
\begin{split}
\langle v, \phi \rangle_{L_{2}} &= \lim_{n \to \infty} \langle A u_n, \phi \rangle_{L_{2}}\\
&= \lim_{n \to \infty} \langle  u_n, A^*\phi \rangle_{L_{2}}\\
&= \langle  u , A^*\phi \rangle_{L_{2}}\\
&= \langle  u , (A_1^* + A_2^* \partial_x )\phi \rangle_{L_{2}}\\
&= \langle  A_1 u ,  \phi \rangle_{L_{2}} + \langle A_2 u,  \phi^\prime \rangle_{L_{2}},
\end{split}
\end{equation*}
with 
\begin{align*}
A_1 &= -e^{tL\partial_x}n''(e^{-tL\partial_x}y) (e^{-tL\partial_x} y_{x}) e^{-tL\partial_x},\\
A_2 &= -e^{tL\partial_x}n'(e^{-tL\partial_x}y)e^{-tL\partial_x},
\end{align*}
both self-adjoint and bounded on \(L_{2}\). Therefore,
\[
(A_2 u)^\prime = v - A_1 u \in L_{2},
\]
\(u \in \dom(A(t,y))\), and \(A(t,y)\) is closed. In order to prove that $A(t,y)$ is quasi-m-accretive, we define 
\begin{equation*}
\begin{split}
Gu & = (e^{tL\partial_x}n'(e^{-tL\partial_x}y)e^{-tL\partial_x}u)_{x}-e^{tL\partial_x}n''(e^{-tL\partial_x}y)e^{-tL\partial_x}y_{x}e^{-tL\partial_x}u,\\
G_{0}u &= -(e^{tL\partial_x}n'(e^{-tL\partial_x}y)e^{-tL\partial_x}u)_{x},
\end{split}
\end{equation*}
with dense domain
\[
\{u\in L_{2} \colon e^{tL\partial_x}n'(e^{-tL\partial_x}y)e^{-tL\partial_x}u\in H^{1}\}. 
\]
The density of $C_{c}^{\infty}$ in $L_{2}$ then implies that $A(t,y)=G$.
Recall that \(e^{\pm t\partial_x L}\) is unitary on \(H^r\), for all \(r \in \R\), and that \(H^{r} \hookrightarrow BC\) for \(r > \frac{1}{2}\). Hence, for any fixed \(y \in B_R\),
\begin{equation}\label{disspative relation}
\begin{aligned}
&\|\partial_x (n'(e^{-tL\partial_x}y))\|_{L_{\infty}}
=\|n''(e^{-tL\partial_x} y) e^{-tL\partial_x} y_{x}\|_{L_{\infty}}\\
&\leq  \|(n''(e^{-tL\partial_x} y) - n''(0)) e^{-tL\partial_x} y_{x}\|_{L_{\infty}} + \|n''(0) e^{-tL\partial_x} y_{x}\|_{L_{\infty}}\\
&\lesssim  \|(n''(e^{-tL\partial_x} y) - n''(0))\|_{H^{s-1}} \|y_{x}\|_{H^{s-1}} + \| y_{x}\|_{H^{s-1}}\\
&\lesssim  R + (1+R)^{s-\frac{3}{2}}R^2,
\end{aligned}
\end{equation}
where we have applied Theorem~\ref{composition theorem} to $n''(\cdot)-n''(0)$. With \eqref{disspative relation} and using the skew-adjointness of \(\partial_x \),  quasi-accretiveness of both $G$ and $G_{0}$ can be proved using integration by parts (this is structurally equivalent to proving quasi-accretiveness of \(u \phi u_x\) for  a well-behaved function \(\phi\)).\\

We now show that \(G\) and \(G_0\) are closed and adjoints of each other.
For closedness, this is analogous to the above proof of that \(A\) is closed.
To see that $G_{0}$ is the adjoint of $G$ in $L_{2}$, consider $v\in \mathrm{dom}(G_{0})$ and $\phi\in C_{c}^{\infty}\subset \mathrm{dom}(G)$. One then has
\begin{equation*}
\begin{split}
\int_{\R}\phi G^{*}v \, \diff x &= \int_{\R}G\phi v\, \diff x,\\
&=  \int_{\R}(e^{tL\partial_x}n'(e^{-tL\partial_x}y)e^{-tL\partial_x}\phi)_{x}v  \, \diff x\\ 
&\quad - e^{tL\partial_x}n''(e^{-tL\partial_x}y)e^{-tL\partial_x}y_{x}e^{-tL\partial_x}\phi v \, \diff x\\
&=\int_{\R}(e^{tL\partial_x}n'(e^{-tL\partial_x}y)e^{-tL\partial_x}\phi_{x})v \, \diff x\\
&=\int_{\R}(e^{tL\partial_x}n'(e^{-tL\partial_x}y)e^{-tL\partial_x}v)\phi_{x} \, \diff x,\\
\end{split}
\end{equation*}
which implies that $v\in \mathrm{dom}(G_{0})$. The density of $C_{c}^{\infty}$ in \(L_{2}\) directly yields that $G^{*}\subset G_{0}$.  To obtain the opposite inclusion, note that just as in the proof of Lemma~\ref{prop LD}, for any $u\in \mathrm{dom}(G_{0})$ there is a sequence of smooth functions $u_{k}$ converging to  $u$ in \(L_{2}\), such that
\begin{equation}\label{appro seq G}
(e^{tL\partial_x}n'(e^{-tL\partial_x}y)e^{-tL\partial_x}u_{k})_{x} \stackrel{L_{2}}{\to} (e^{tL\partial_x}n'(e^{-tL\partial_x}y)e^{-tL\partial_x}u)_{x},
\end{equation}
as \(k \to \infty\). As above, write $f_{k}$ for the convolution \(f*\varrho_{k}\), where \(\varrho_k\) is a standard mollifier. It is then clear (cf. the proof of Lemma~\ref{prop LD}) that 
\[
Q(\partial_{x}) (f \ast \varrho_k) = (Q(\partial_{x}) f) \ast \varrho_k \stackrel{L_{2}}{\to}  Q(\partial_{x}) f,
\] 
for any Fourier multiplier operator \(Q(\partial_{x})\) for which \(Q(\partial_{x}) f \in L_{2}\). If \(n\) is a bounded function, one therefore immediately obtains the required convergence \(n Q(\partial_{x}) (f \ast \varrho_k) \stackrel{L_{2}}{\to}  n Q(\partial_{x}) f\). In the case of an operator \(\partial_{x} [n Q(\partial_{x})]\) as in \eqref{appro seq G}, where \(n\) is a bounded function such that \(\partial_{x}n \in L_{\infty}\), one notes that
\[
\partial_{x} [n Q(\partial_{x})] = (\partial_{x} n) Q(\partial_{x}) + n \partial_{x} Q(\partial_{x}),
\]
and both terms which are of the form \(\tilde n \tilde Q(\partial_{x})\). This argument is valid for any fixed \(y \in B_R \subset H^s\). Thus, for $\color{black} v\in \mathrm{dom}(G_{0})$, there exists a sequence $\{v_{k}\}$ satisfying \eqref{appro seq G} such that
\begin{equation}\label{eq:limit equation}
\begin{aligned}
&\int_{\R}(Gu) v_{k}\\
&=\int_{\R}\big(e^{tL\partial_x}n'(e^{-tL\partial_x}y)e^{-tL\partial_x}u)_{x}-e^{tL\partial_x}n''(e^{-tL\partial_x}y)e^{-tL\partial_x}y_{x}e^{-tL\partial_x}u\big) v_{k}\, \diff x\\
&=-\int_{\R}\big[(e^{tL\partial_x}n'(e^{-tL\partial_x}y)e^{-tL\partial_x}u)(v_{k})_{x}\\
&\quad+(e^{tL\partial_x}n''(e^{-tL\partial_x}y)e^{-tL\partial_x}y_{x}e^{-tL\partial_x}u)v_{k}\big] \, \diff x\\
&=-\int_{\R}u\Big[e^{tL\partial_x}n'(e^{-tL\partial_x}y)e^{-tL\partial_x}(v_{k})_{x}\\
&\quad+(e^{tL\partial_x}n''(e^{-tL\partial_x}y)e^{-tL\partial_x}y_{x}e^{-tL\partial_x}v_{k})\Big] \, \diff x\\
&=-\int_{\R}u\big(e^{tL\partial_x}n'(e^{-tL\partial_x}y)e^{-tL\partial_x}v_{k}\big)_{x}\, \diff x,
\end{aligned}
\end{equation}
Taking the limit with respect to $k$  in \eqref{eq:limit equation}, we deduce for all $u\in \mathrm{dom}(G)$ that
\begin{equation*}
\int_{\R}(Gu) v=\int_{\R}u(G_{0}v),
\end{equation*}
which means that $G_{0}\subset G^{*}$. Therefore, $G_{0}=G^{*}$. By \cite[Corollary 4.4]{Pazy83}, $G$ and hence $A(t,y)$ are quasi-m-accretive.\\

\end{proof}

Denote by 
\[
[T_{1},T_{2}]=T_{1}T_{2}-T_{2}T_{1}
\] 
the commutator of two general operators $T_{1}$ and $T_{2}$. Since $\partial_x$ and $\mathit{L}$ are both multiplier operators, clearly $[\partial_x,\mathit{L}]=0$ in a Sobolev setting. Let
\begin{equation}\label{eq:B}
B(t,y) = \Lambda^s (A(t,y))\Lambda^{-s}-A(t,y) = [\Lambda^{s}, A(t,y)]\Lambda^{-s},
\end{equation}
with \(\Lambda^s = (1 - \partial_x^2)^{\frac{s}{2}}\), and \(A(t,y)\) defined as in~\eqref{eq:A}. We prove the Lipschitz continuity of the operators $A(t,\cdot)$ and $B(t,\cdot)$.

\begin{lemma}\label{lemma:Lipschitz}
The operator \(B(t,y)\) from~\eqref{eq:B} is bounded in \(\mathcal{B}(X)\), with
\begin{equation}\label{eq:B Lipschitz}
\|(B(t,y)-B(t,z))w\|_{X} \lesssim \|y-z\|_{Y}\|w\|_{X},
\end{equation}
uniformly for all \(y,z\in B_R \subset Y\) and all \(w\in X\).
The estimate \eqref{eq:B Lipschitz} holds if we replace \(A(t,\cdot)\) by \(B(t,\cdot)\) and interchange the norms in $X$ and $Y$ spaces on the right hand side.
\end{lemma}

\begin{proof}
Because Fourier multipliers commute, one has
\begin{equation*}
[\Lambda^{s},A(y)]=e^{tL\partial_x}[\Lambda^{s},n'(e^{-tL\partial_x}y)]e^{-tL\partial_x}\partial_x,
\end{equation*}
{\color{black} and with classical commutator estimates  (cf. \cite{MR2333212,MR0407477})} that
\[
\|[\Lambda^{s},n'(e^{-tL\partial_x}y)]\Lambda^{1-s}\| \lesssim \|\partial_x(n'(e^{-tL\partial_x}y))\|_{H^{s-1}},
\]
for all \(y \in B_R\). With Theorem~\ref{composition theorem}, one {\color{black} further} has that
\begin{equation}\label{eq:split n}
\begin{aligned}
&\|\partial_x(n'(e^{-tL\partial_x}y))\|_{H^{s-1}}\\ 
&\lesssim \|(n''(e^{-tL\partial_x}y)-n''(0)) e^{-tL\partial_x}y_{x}\|_{H^{s-1}} + \|n''(0) e^{-tL\partial_x}y_{x}\|_{H^{s-1}}\\
&\lesssim \|n''(e^{-tL\partial_x}y)-n''(0) \|_{H^{s-1}}\|e^{-tL\partial_x}y_{x}\|_{H^{s-1}} + |n''(0) |\|y\|_{H^{s}}\\
&\lesssim (1+R)^{s-\frac{3}{2}}R^{2}+R,
\end{aligned}
\end{equation}
where all estimates depend upon the radius of the ball \(B_R\) in which \(y\) lies, and the final estimate also on the nonlinearity \(n\). Thus, for any $z\in L_{2}$, we have 
\begin{align*}
\|B(t,y)z\|_{L_{2}} &= \|[\Lambda^{s},n'(e^{-tL\partial_x}y)]\Lambda^{1-s}\Lambda^{s-1}\partial_x\Lambda^{-s}z\|_{L_{2}}\\
&\leq  \|[\Lambda^{s},n'(e^{-tL\partial_x}y)]\Lambda^{1-s}\|\|\Lambda^{s-1}\partial_x\Lambda^{-s}z\|_{L_{2}}\\
&\lesssim \left((1+R)^{s-\frac{3}{2}}R^{2}+R\right)\|z\|_{L_{2}},
\end{align*}
whence $B(t,y)$ is bounded on $L_{2}$, uniformly for \(y \in B_R\).

To prove the Lipschitz continuity in \(y\), notice that for any $y,z\in B_R$ and $w \in X$, one has the uniform estimate
\begin{align*}
&\|B(y)w-B(z) w\|_{L_{2}}\\ 
&\lesssim \|[\Lambda^{s},n'(e^{-tL\partial_x}y)-n'(e^{-tL\partial_x}z)]\Lambda^{1-s}\| \|\Lambda^{s-1}\partial_x\Lambda^{-s} w\|_{L_{2}}\\
&\lesssim \|\partial_x(n'(e^{-tL\partial_x}y)-n'(e^{-tL\partial_x}z))\|_{H^{s-1}}\|w\|_{L_{2}}\\
&\lesssim \|n'(e^{-tL\partial_x}y)-n'(e^{-tL\partial_x}z)\|_{H^{s}}\|w\|_{L_{2}}.
\end{align*}
Appealing to Theorem~\ref{composition theorem}, one furthermore estimates
\begin{align*}
&\|n'(e^{-tL\partial_x}y)-n'(e^{-tL\partial_x}z)\|_{H^{s}}\\
&= \|\int^{1}_{0}n''(e^{-tL\partial_x}(z+t(y-z)))e^{-tL\partial_x}(y-z)\,\, \diff t\|_{H^{s}}\\
&\lesssim \left(R(1+R)^{s-\frac{3}{2}}+1\right)\|y-z\|_{H^{s}},
\end{align*}
where we have used the same splitting of \(n^{\prime\prime}\) as in \eqref{eq:split n}. 
Thus, $B(y)$ satisfies condition (ii) in Theorem~\ref{Kato's theorem} for $y\in B_R \subset H^{s}$,  $s>\frac{3}{2}$. The proof of~\eqref{eq:B Lipschitz} with $A(t,y)$ substituted for \(B(t,y)\) is structurally similar, but easier, than the above proof and we omit the details.
\end{proof}

Since the operator $A(t,y)$ relies on $t$, besides the assumptions in Theorem \ref{Kato's theorem} one also needs to verify the continuity of the map $t\mapsto A(t,y)\in \mathcal{B}(Y,X)$ for each $y\in B_R \subset Y$.  As remarked in \cite{MR0407477}, it however suffices to prove that $t\mapsto A(t,y)$ is strongly continuous.

\begin{lemma}\label{lemma:strongly continuous}
The map $t\mapsto A(t,y)\in \mathcal{B}(Y,X)$ is strongly continuous. 
\end{lemma}

\begin{proof}
Since $\partial_x$ is bounded from $Y$ to $X$ and $e^{-tL\partial_x}$ is a strongly continuous unitary group on both $X$ and $Y$, it {\color{black} is enough} to prove that the multiplication operator $n'(e^{-tL\partial_x}y)-n'(0)\in \mathcal{B}(Y,X)$ is strongly continuous in $t$. In view of Theorem~\ref{composition theorem}, that map is continuous even in norm, since \(t \mapsto e^{-tL\partial_x}y\) is continuous \(\R \to H^{s}\) for \(s>\frac{3}{2}\).
\end{proof}

We are now ready to prove the main theorem for initial data \(u_0 \in H^s(\mathbb R)\). 
\begin{proof}[Proof of Theorem~\ref{main theorem}]
Based on Lemmata~\ref{lemma:quasi-m}, \ref{lemma:Lipschitz} and~\ref{lemma:strongly continuous}, we may apply Theorem~\ref{Kato's theorem} to find a solution \(v\) to equation \eqref{transed-eq} in the solution class \(C([0,T); H^{s}(\mathbb{R}))\cap C^{1}([0,T); L_{2}(\mathbb{R}))\). Because \(H^{s-1}(\mathbb R)\), $s>\frac{3}{2}$, is an algebra, and 
\[
v\mapsto e^{tL\partial_x}n'(e^{-tL\partial_x}v)e^{-tL\partial_x}\partial_x v
\] 
maps \(H^s(\mathbb R)\) continuously into \(H^{s-1}(\mathbb R)\) one, however, sees that 
\[
v_{t} = -A(t,v)v=-e^{tL\partial_x}n'(e^{-tL\partial_x}v)e^{-tL\partial_x}\partial_x v \in H^{s-1}(\mathbb R).
\]
Recall that $\{e^{-tL\partial_x}\}$ forms a unitary group on $H^{s}$, for any $s\geq 0$. Hence, \(v \in C^{1}([0,T); H^{s-1}(\mathbb{R}))\).  Also, since \([v_0 \mapsto v]\) is continuous \(\color{black} H^s(\mathbb{R}) \to C([0,T),H^s(\mathbb{R}))\), and \(\partial_x\) maps \(H^s(\mathbb R)\) continuously into \(H^{s-1}(\mathbb R)\), the same argument can be used to conclude that 
\[
[v_0 \mapsto v] \in C(H^s(\mathbb{R}),C^1([0,T),H^{s-1}(\mathbb{R}))).
\]
The transformation \eqref{pre-transf} gives a solution $u$ of~\eqref{eq:main} in the class of $C([0,T);H^{s})$ and \([u_0 \mapsto u] \in C(H^s(\mathbb{R}),C([0,T),H^s(\mathbb{R}))\). Now,
\begin{equation}\label{time derivative for u}
u_{t}(t,x)=-e^{-tL\partial_x}L\partial_x v(t,x)+e^{-tL\partial_x}v_{t}(t,x),
\end{equation}
so \(u_t\) in general inherits its smoothness from the minimal regularity of \(v_t\) and  \(L\partial_x v\). By assumption \ref{main assum}, \(L\) is a \(l\)th-order operator, whence we deduce that 
\begin{equation*}
u_{t}(t,x)\in C([0,T);H^{s-\max\{1, l+1\}}).
\end{equation*}
A technique similar to the one above gives the continuous dependence upon initial data $u_{0}\in H^{s}$. 
\end{proof}


\section{Well-posedness for periodic initial data}\label{the case of periodic initial data}
This section focuses on  the  Cauchy problem \eqref{eq:main} in the periodic setting \(\color{black} H^s(\T)\). We first introduce some {\color{black} notation}. For a function $f$  defined on $\mathbb{R}^{n}$ or the $n\textrm{-}$dimensional torus $\T^{n}$,  we define the {\color{black} \(m\)th-order} difference about $f$ by
\begin{equation*}
\begin{split}
\Delta_{h}^{1}f(x)&=f(x+h)-f(x),\quad h,x\in \mathbb{F}^{n},\\
\Delta_{h}^{m}f(x)&=\Delta_{h}^{1}(\Delta_{h}^{m-1}f)(x),\quad m=2,3,...,\quad h,x\in \mathbb{F}^{n},\\
\end{split}
\end{equation*}
where $\mathbb{F}\in \{\R,\T\}$. 
It is easy to verify that
\begin{equation}\label{2nd difference}
2(\Delta_{h}f)(x)=(\Delta_{2h}f)(x)-(\Delta_{h}^{2}f)(x).
\end{equation}
Denote by $\N$ and $\N_{0}$ the positive and nonnegative integers, respectively. For a multi-index $\alpha =(\alpha_{1},\cdot\cdot\cdot,\alpha_{n}),\, \alpha_{j}\in \N_{0}$, we use $\partial^{\alpha}=\frac{\partial^{|\alpha|}}{\partial^{\alpha_{1}}_{x_{1}}\cdot\cdot\cdot\partial^{\alpha_{n}}_{x_{n}}}$
to denote multi-index derivatives with the index $\alpha$, where $|\alpha| =\alpha_{1}+\cdot\cdot\cdot+\alpha_{n}$.
For $s>0$, we use the decomposition $s=[s]^{-}+\{s\}^{+}$ where $[s]^{-}$ is  an integer and $\{s\}^{+}\in (0,1]$. 

\medskip

There is a general characterization for Besov spaces based on spectral decomposition and distribution theory (cf. \cite{MR730762}) but the derivative--difference characterization is enough here since only the classical, normed, Besov spaces  $B^{s}_{pq}(\mathbb{T}^{n})$, $s>0,\, p\in (1,\infty),\,q\in [1,\infty]$, {\color{black} appear in our work}.\footnote{Note that the Besov spaces $B^{s}_{pq}(\mathbb{R}^{n})$  or $B^{s}_{pq}(\mathbb{T}^{n})$ are only quasi-normed spaces in general; however, they are normed spaces for the indices $s,p,q$ in the current setting. {\color{black} The} case $p=1$ is excluded since in this case the composition theorem is not expected to hold for spaces over a torus, as indicated in Theorem \ref{composition theorem} for spaces over the whole space.} 
With the derivative--difference characterization, the Besov spaces $B^{s}_{pq}(\mathbb{T}^{n}),\, q\in [1,\infty)$,   have {\color{black} norms} (cf. \cite{SmTr87})
\begin{equation}\label{Besov norms on torus}
\|f\|_{B^{s}_{pq}(\mathbb{T}^{n})} =\|f\|_{W^{[s]^{-}}_{p}(\mathbb{T}^{n})}+\sum_{|\alpha|=[s]^{-}}\left(\int_{\mathbb{T}^{n}}|h|^{-\{s\}^{+}q}\|\Delta_{h}^{2}\partial^{\alpha}{\color{black} f}\|^{q}_{L_{p}(\mathbb{T}^{n})}\frac{\mathrm{d}h}{|h|^{n}}\right)^{\frac{1}{q}}.
\end{equation}
{\color{black} where $W^{m}_{p}(\mathbb{T}^{n})$, $m\in \mathbb{N}_{0}$, are the standard  Sobolev spaces with norms}
\begin{equation*}
\|f\|_{W^{m}_{p}(\mathbb{T}^{n})}=\sum_{|\alpha|\leq m}\|\partial^{\alpha}f\|_{L_{p}(\mathbb{T}^{n})},
\end{equation*}
{\color{black}  and \(L_{p}(\mathbb{T}^{n})\) is defined in \eqref{periodic torus norm identify}. In the case when $q=\infty$, the above norm must be modified to}
\begin{equation}\label{Besov norms on torus q infty}
\|f\|_{B^{s}_{p\infty}(\mathbb{T}^{n})} =\|f\|_{W^{[s]^{-}}_{p}(\mathbb{T}^{n})}+\sum_{|\alpha|=[s]^{-}}\sup_{0\neq h\in \T^{n}}|h|^{-\{s\}^{+}}\|\Delta_{h}^{2}\partial^{\alpha} {\color{black} f}\|_{L_{p}(\mathbb{T}^{n})}.
\end{equation}
{\color{black} A function $f$ defined on $\T^{n}$ is  naturally identified with a $2\pi$-periodic function on $\R^{n}$, and we denote this other function still by $f$. With that} identification, one can define the Lebesgue measure on $\T^{n}$ and set (cf. \cite{katznelson2004introduction})
\begin{equation}\label{periodic torus norm identify}
\|f\|_{L_{p}(\mathbb{T}^{n})}=\left(\int_{-\pi}^{\pi}\dots\int_{-\pi}^{\pi}|f|^{q}\, \diff x\right)^{\frac{1}{q}}.
\end{equation}
Our composition theorem for periodic Besov spaces relies on Theorem \ref{composition theorem} for spaces over $\R^{n}$. {\color{black} The Besov spaces $B^{s}_{p,q}(\R^{n})$, $s>0,\,p\in (1,\infty),\, q \in [1,\infty]$, can be defined just as in  \eqref{Besov norms on torus} and \eqref{Besov norms on torus q infty} by replacing \(\T^n\) with \(\R^n\) \cite{MR730762}. The following remark follows from a direct calculation.} 
\begin{remark}\label{index s and h}
{\color{black} Let $0 < \delta \ll 1$. Replacing the integral domain \(\T^n\) or $\R^{n}$ by the domain $|h|<\delta$ in the norms \eqref{Besov norms on torus} and \eqref{Besov norms on torus q infty} of the Besov spaces $B^{s}_{pq}$ over \(\T^n\) and \(\R^n\), respectively,  yields an equivalent norm on the same Besov spaces.}
\end{remark}
We proceed by {\color{black} establishing} a relation between periodic and non-periodic Besov spaces. This relation and the techniques used in its proof are key ingradients for the proof of our composition theorem for Besov spaces on a torus.
\begin{lemma}[{\color{black} A localizing property}]\label{per nonper control}
Let $s>0$, $p\in (1,\infty)$ and $q\in [1,\infty]$. For any $\varrho\in C_{c}^{\infty}(\R^{n})$, {\color{black} one has
\begin{equation*}
\|\varrho f\|_{B^{s}_{pq}(\R^{n})}\lesssim \|f\|_{B^{s}_{pq}(\T^{n})},
\end{equation*}
uniformly for all \(f\in B^{s}_{pq}(\T^{n}).\)}
\end{lemma}
\begin{proof}
{\color{black} Let \(\varrho \in C_c^\infty(\R^n)\) be as in the lemma.} Since $f\in B^{s}_{pq}(\T^{n})$ can be identified with a $2\pi$-periodic function  on $\R^{n}$ (still denoted by $f$) as mentioned above, the finiteness of $\varrho f$ (as a function on $\R^{n}$) under the norm $\|\cdot\|_{B^{s}_{pq}(\R^{n})}$ follows from the {\color{black} smoothness and compact support} of $\varrho$. {\color{black} In the following, we give the full details for the case $1\leq q<\infty$; the proof for $q=\infty$ then follows with minor changes using the same procedure. We also consider separately the cases when \(s\) is, or is not, an integer.}

\textit{The case when $s$ is not an integer.} 
By {\color{black} the} Leibniz rule for differentiation, it suffices to consider the case when $s\in (0,1)$, where $[s]^{-}=0$ and $\{s\}^{+}=s$.  {\color{black} By  Remark~\ref{index s and h},} we {\color{black} then} have
\begin{equation*}\label{0 order reduction}
\begin{split}
\|f\varrho\|_{B^{s}_{pq}(\mathbb{R}^{n})}
&\lesssim\|f\varrho\|_{L_{p}(\mathbb{R}^{n})}+\left(\int_{|h|<\delta}|h|^{-s q}\|\Delta_{h}^{2} {\color{black} (f\varrho)} \|^{q}_{L_{p}(\mathbb{R}^{n})}\frac{\mathrm{d}h}{|h|^{n}}\right)^{\frac{1}{q}},
\end{split}
\end{equation*}
for {\color{black} a fixed} $\delta \in (0,\frac{1}{2})$. {\color{black} Given \(\varrho\) there exists \(N\) such that  $\mathrm{supp}(\varrho)=[a,b]^{n}\subset[-N\pi,N\pi]^{n}$. Then}
\begin{equation}\label{LP control}
\|f\varrho\|_{L_{p}(\mathbb{R}^{n})}\leq \|f\varrho\|_{L_{p}([-N\pi,N\pi]^{n})}\leq N^{n}\|\varrho\|_{L_{\infty}(\R^{n})}\|f\|_{L_{p}(\T^{n})},
\end{equation}
and 
\begin{equation}\label{compact support 2nd difference}
\color{black}
\|\Delta_{h}^{2}(f\varrho)\|_{L_{p}(\mathbb{R}^{n})}=\|\Delta_{h}^{2}(f\varrho)\|_{L_{p}([a-2\delta,b+2\delta]^{n})},
\end{equation}  
where we used the identification between functions on $\T^{n}$ and $2\pi$-periodic functions on $\R^{n}$. In view of  {\color{black} (this can be seen by rewriting)}
\begin{equation}\label{2 order iteration}
\Delta^{2}_{h}(f\varrho)(x)=(\Delta^{2}_{h}f)(x)\varrho(x+h)+f(x+2h)(\Delta^{2}_{h}\varrho)(x)+(\Delta_{2h}f)(x)(\Delta_{h} \varrho)(x),
\end{equation}
we derive from \eqref{LP control} and \eqref{compact support 2nd difference} that
\begin{equation}\label{3 terms decom}
\color{black}
\begin{split}
\|f\varrho\|_{B^{s}_{pq}(\mathbb{R}^{n})}&\lesssim N^{n}\|\varrho\|_{L_{\infty}(\R^{n})}\|f\|_{L_{p}(\T^{n})}\\
&\quad +\left(\int_{|h|<\delta}|h|^{-s q}\|(\Delta^{2}_{h}f) \varrho(\cdot+h)\|^{q}_{L_{p}([a-2\delta,b+2\delta]^{n})}\frac{\mathrm{d}h}{|h|^{n}}\right)^{\frac{1}{q}}\\
&\quad +\left(\int_{|h|<\delta}|h|^{-s q}\|f(\cdot+2h)\Delta^{2}_{h}\varrho\|^{q}_{L_{p}([a-2\delta,b+2\delta]^{n})}\frac{\mathrm{d}h}{|h|^{n}}\right)^{\frac{1}{q}}\\
&\quad +\left(\int_{|h|<\delta}|h|^{-s q}\|\Delta_{2h}f \Delta_{h} \varrho \|^{q}_{L_{p}([a-2\delta,b+2\delta]^{n})}\frac{\mathrm{d}h}{|h|^{n}}\right)^{\frac{1}{q}}\\
&=:N^{n}\|\varrho\|_{L_{\infty}(\R^{n})}\|f\|_{L_{p}(\T^{n})}+T_{1}+T_{2}+T_{3}.
\end{split}
\end{equation}
By the periodicity of $\Delta^{k}_{h}f$, $k\geq 1$, we estimate $T_{1}$ as follows
\begin{equation*}\label{T1 esti}
\begin{split}
T_{1}
&\leq \|\varrho\|_{L_{\infty}(\R^{n})}\left(\int_{|h|<\delta}|h|^{-s q}\|\Delta^{2}_{h} {\color{black}f}\|^{q}_{L_{p}([a-2\delta,b+2\delta]^{n})}\frac{\mathrm{d}h}{|h|^{n}}\right)^{\frac{1}{q}}\\
&\leq (N+2)^{n}\|\varrho\|_{L_{\infty}(\R^{n})}\left(\int_{\T^{n}}|h|^{-s q}\|\Delta^{2}_{h}f\|^{q}_{L_{p}(\T^{n})}\frac{\mathrm{d}h}{|h|^{n}}\right)^{\frac{1}{q}},
\end{split}
\end{equation*}
{\color{black} where in the last inequality we used the fact that $[a-2\delta,b+2\delta]^{n}\subset [-(N+2)\pi,(N+2)\pi]^{n}$ and the identification between functions on $\T^{n}$ and the $2\pi$-periodic functions on $\R^{n}$.}

Before estimating the term $T_{2}$, we observe from mean value theorem for scalar \cite{hubbard1999vector} and vector-valued  \cite{lang2013undergraduate} multi-variable functions  that
\begin{equation}\label{esti Delta 2 g}
\begin{split}
|\Delta^{2}_{h}\varrho|&=\left|[\varrho(x+2h)-\varrho(x+h)]-[\varrho(x+h)-\varrho(x)]\right|\\
&=|\nabla\varrho(\xi_{1})\cdot h-\nabla\varrho(\xi_{2})\cdot h|\\
&=|\int_{0}^{1}\nabla\Big((\nabla\varrho)(t\xi_{1}+(1-t)\xi_{2})\Big)\, \diff t\cdot(\xi_{1}-\xi_{2})||h|\\
&\leq 2\sup_{t\in[0,1]}|\mathrm{Hess}(\varrho)(t\xi_{1}+(1-t)\xi_{2})||h|^{2},
\end{split}
\end{equation}
where $\nabla$ and $\mathrm{Hess}$ denote the gradient and the Hessian Matrix, respectively; $\xi_{1}$ lies on the segment with endpoints $x+2h $ and $x+h$, $\xi_{2}$ lies on the segment with endpoints $x+h$ and $x$. Then, $T_{2}$ can be estimated as
\begin{equation*}\label{T2 esti}
\begin{split}
T_{2}
&\leq 2\|\mathrm{Hess}(\varrho)\|_{L_{\infty}(\R^{n})}\left(\int_{|h|<\delta}|h|^{(2-s) q}\|{\color{black} f(\cdot+2h)}\|^{q}_{L_{p}([a-2\delta,b+2\delta]^{n})}\frac{\mathrm{d}h}{|h|^{n}}\right)^{\frac{1}{q}}\\
&\leq 2\|\mathrm{Hess}(\varrho)\|_{L_{\infty}(\R^{n})}\|{\color{black} f}\|_{L_{p}([a-4\delta,b+4\delta]^{n})}\left(\int_{|h|<\delta}|h|^{(2-s) q}\frac{\mathrm{d}h}{|h|^{n}}\right)^{\frac{1}{q}}\\
& {\color{black} \lesssim}  (N+4)^{n} \|\mathrm{Hess}(\varrho)\|_{L_{\infty}(\R^{n})}\| {\color{black} f}\|_{L_{p}(\T^{n})},
\end{split}
\end{equation*} 
{\color{black} where we have used that \(\int_{|h|<\delta}|h|^{(2-s) q}\frac{\mathrm{d}h}{|h|^{n}}\) is finite and absolutely bounded for $s=\{s\}^{+}\in (0,1)$ and \(\delta \in (0, \frac{1}{2})\). In the first inequality above we applied \eqref{esti Delta 2 g}, in the second the variable substitution $x+2h\mapsto x$ and in the last that $[a-4\delta,b+4\delta]^{n}\subset [-(N+4)\pi,(N+4)\pi]^{n}$.} The term $T_{3}$ can be similarly estimated:
\begin{equation*}\label{T3 esti}
\begin{split}
T_{3}
&\leq \|\nabla\varrho\|_{L_{\infty}(\R^{n})}\left(\int_{|h|<\delta}|h|^{(1-s) q}\|{\color{black} f(\cdot+2h)}\|^{q}_{L_{p}([a-2\delta,b+2\delta]^{n})}\frac{\mathrm{d}h}{|h|^{n}}\right)^{\frac{1}{q}}\\
&\quad+\|\nabla\varrho\|_{L_{\infty}(\R^{n})}\left(\int_{|h|<\delta}|h|^{(1-s) q}\|{\color{black} f}\|^{q}_{L_{p}([a-2\delta,b+2\delta]^{n})}\frac{\mathrm{d}h}{|h|^{n}}\right)^{\frac{1}{q}}\\
&\leq 2\|\nabla\varrho\|_{L_{\infty}(\R^{n})}\|{\color{black} f}\|_{L_{p}([a-4\delta,b+4\delta]^{n})}\left(\int_{|h|<\delta}|h|^{(1-s) q}\frac{\mathrm{d}h}{|h|^{n}}\right)^{\frac{1}{q}}\\
&{\color{black} \lesssim} (N+4)^{n}\|\nabla\varrho\|_{L_{\infty}(\R^{n})} \|{\color{black} f}\|_{L_{p}(\T^{n})},
\end{split}
\end{equation*} 
{\color{black} where again the estimate is absolute because \(\delta \in (0,\frac{1}{2})\) and  $s=\{s\}^{+}\in (0,1)$. In the first inequality we relied on the mean-value theorem for $\varrho$, in the second inequality on the variable substitution $x+2h\mapsto x$, and in the third on that $[a-4\delta,b+4\delta]^{n}\subset [-(N+4)\pi,(N+4)\pi]^{n}$. By} inserting above estimates for $T_{1}$, $T_{2}$ and $T_{3}$ into \eqref{3 terms decom}, we have 
\begin{equation*}
\begin{split}
\|f\varrho\|_{B^{s}_{pq}(\mathbb{R}^{n})}
&\lesssim_{s,p,q,n,\varrho}\left(\|f\|_{L_{p}(\mathbb{R}^{n})}+\left(\int_{\T^{n}}|h|^{-s q}\|\Delta^{2}_{h}f\|^{q}_{L_{p}(\T^{n})}\frac{\mathrm{d}h}{|h|^{n}}\right)^{\frac{1}{q}}\right)
\end{split}
\end{equation*}
{\color{black} which concludes the proof for the case when $s$ is a non-integer.}\\

\textit{The case when $s$ is an integer.} As above, we still {\color{black} pick} $\delta\in (0,\frac{1}{2})$. For an integer $s$, we have $[s]^{-}=s-1$ and $\{s\}^{+}=1$.  As mentioned above, {\color{black} by Leibniz's rule} it suffices to consider the case $[s]^{-}=0$ and $\{s\}^{+}=1$.  Similar {\color{black} to} \eqref{3 terms decom} we deduce from Remark~\ref{index s and h} that
\begin{equation}\label{0 order reduction non-int s}
\begin{split}
\|f\varrho\|_{B^{s}_{pq}(\mathbb{R}^{n})}
&\lesssim N^{n}\|\varrho\|_{L_{\infty}(\R^{n})}\|f\|_{L_{p}(\T^{n})}\\
&\quad +\left(\int_{|h|<\delta}|h|^{-q}\|(\Delta^{2}_{h}f) {\color{black} \varrho(\cdot+h)}\|^{q}_{L_{p}([a-2\delta,b+2\delta]^{n})}\frac{\mathrm{d}h}{|h|^{n}}\right)^{\frac{1}{q}}\\
&\quad +\left(\int_{|h|<\delta}|h|^{-q}\|{\color{black} f(\cdot+2h)}\Delta^{2}_{h}\varrho\|^{q}_{L_{p}([a-2\delta,b+2\delta]^{n})}\frac{\mathrm{d}h}{|h|^{n}}\right)^{\frac{1}{q}}\\
&\quad +\left(\int_{|h|<\delta}|h|^{-q}\|\Delta_{2h}f\Delta_{h} \varrho\|^{q}_{L_{p}([a-2\delta,b+2\delta]^{n})}\frac{\mathrm{d}h}{|h|^{n}}\right)^{\frac{1}{q}}\\
&=:N^{n}\|\varrho\|_{L_{\infty}(\R^{n})}\|f\|_{L_{p}(\T^{n})}+S_{1}+S_{2}+S_{3}.
\end{split}
\end{equation}
{\color{black} Analogous to the estimates for $T_{1}$ and $T_{2}$, we get control of $S_{1}$ and $S_{2}$.}
\begin{align*}
S_{1}&\leq (N+2)^{n}\|\varrho\|_{L_{\infty}(\R^{n})}\left(\int_{\T^{n}}|h|^{-q}\|\Delta^{2}_{h}f\|^{q}_{L_{p}(\T^{n})}\frac{\mathrm{d}h}{|h|^{n}}\right)^{\frac{1}{q}},\\
S_{2}& {\color{black} \lesssim} (N+4)^{n} \|\mathrm{Hess}(\varrho)\|_{L_{\infty}(\R^{n})}\|{\color{black} f}\|_{L_{p}(\T^{n})},
\end{align*}
{\color{black} where in the second estimate we have used that $\int_{|h|<\delta}|h|^{q}\frac{\mathrm{d}h}{|h|^{n}}$ is absolutely bounded for \(\delta \in (0,\frac{1}{2})\), $q\in [1,\infty)$.} However, the argument for estimating $T_{3}$ is not suitable for  $S_{3}$, since the integration near the origin {\color{black} would lead to a  logarithmic} blow-up and therefore fails to give a finite control for $S_{3}$. Instead, we use the difference iteration formula \eqref{2nd difference} and carefully analyze the {\color{black} integrals over each sub-interval}. The estimate {\color{black} for} $S_{3}$ starts as follows. 
\begin{equation}\label{0 order reduction non-int s}
\begin{split}
S_{3}
&\leq \|\nabla\varrho\|_{L_{\infty}(\R^{n})}\left(\int_{|h|<\delta}\|\Delta_{2h}f\|^{q}_{L_{p}([a-2\delta,b+2\delta]^{n})}\frac{\mathrm{d}h}{|h|^{n}}\right)^{\frac{1}{q}}\\
&\leq \frac{1}{2}\|\nabla\varrho\|_{L_{\infty}(\R^{n})}\left(\int_{|h|<\delta}\|\Delta_{2h}^{2}f\|^{q}_{L_{p}([a-2\delta,b+2\delta]^{n})}\frac{\mathrm{d}h}{|h|^{n}}\right)^{\frac{1}{q}}\\
&\quad+\frac{1}{2}\|\nabla\varrho\|_{L_{\infty}(\R^{n})}\left(\int_{|h|<\delta}\|\Delta_{4h}f\|^{q}_{L_{p}([a-2\delta,b+2\delta]^{n})}\frac{\mathrm{d}h}{|h|^{n}}\right)^{\frac{1}{q}}\\
&=:\frac{1}{2}\|\nabla\varrho\|_{L_{\infty}(\R^{n})}\left(S_{31}+S_{32}\right),
\end{split}
\end{equation}
where the first inequality used the mean value theorem for scalar functions with several variables; the second inequality used the interation formula \eqref{2 order iteration}.
{\color{black} We first make the variable substitution $2h\mapsto \eta$ in $S_{32}$ and divide the interval of integration as}
\begin{equation*}\label{lift order and step}
\begin{split}
{\color{black} (S_{32})^{q}} &=\int_{|\eta|<2\delta}\|\Delta_{2\eta}f\|^{q}_{L_{p}([a-2\delta,b+2\delta]^{n})}\frac{\mathrm{d}\eta}{|\eta|^{n}}\\
&= {\color{black}\left( \int_{|\eta|<\delta} + \int_{\delta<|\eta|<2\delta} \right) \|\Delta_{2\eta}f\|^{q}_{L_{p}([a-2\delta,b+2\delta]^{n})}\frac{\mathrm{d}\eta}{|\eta|^{n}}.}
\end{split}
\end{equation*}
Then \eqref{0 order reduction non-int s} and the above {\color{black} equality for} $S_{32}$ imply that  
\begin{equation}\label{midd contr}
\begin{split}
\Bigg(\int_{|h|<\delta}\|\Delta_{2h}f\|&^{q}_{L_{p}([a-2\delta,b+2\delta]^{n})}\frac{\mathrm{d}h}{|h|^{n}}\Bigg)^{\frac{1}{q}}\\
&\leq  S_{31}+\left(\int_{\delta<|h|<2\delta}\|\Delta_{2h}f\|^{q}_{L_{p}([a-2\delta,b+2\delta]^{n})}\frac{\mathrm{d}h}{|h|^{n}}\right)^{\frac{1}{q}}.
\end{split}
\end{equation}
Inserting \eqref{midd contr} into \eqref{0 order reduction non-int s}, we get
\begin{equation}\label{midd contr 2}
\begin{split}
S_{3}&\leq \|\nabla\varrho\|_{L_{\infty}(\R^{n})}\left(S_{31}+\left(\int_{\delta<|h|<2\delta}\|\Delta_{2h}f\|^{q}_{L_{p}([a-2\delta,b+2\delta]^{n})}\frac{\mathrm{d}h}{|h|^{n}}\right)^{\frac{1}{q}}\right).\\
\end{split}
\end{equation}
{\color{black} To estimate \(S_{31}\), since $q>1$ we again use the substitution $2h\mapsto \eta$ in $S_{31}$ to get}
\begin{equation}\label{esti S31}
\begin{aligned}
{\color{black} (S_{31})^{q}} &=\int_{|\eta|<2\delta}\|\Delta_{\eta}^{2}f\|^{q}_{L_{p}([a-2\delta,b+2\delta]^{n})}\frac{\mathrm{d}\eta}{|\eta|^{n}}\\
&\leq (N+2)^{nq}\int_{\T^{n}} |\eta|^{-q}  \|\Delta_{\eta}^{2}f\|^{q}_{L_{p}(\T^{n})}  \frac{\mathrm{d}\eta}{|\eta|^{n}},
\end{aligned}
\end{equation}
{\color{black} where the factor \(|\eta|^{-q}\) may be trivially introduced since \(|\eta| \leq 1\) and \(q \geq 1\).} For the remaining term in \eqref{midd contr 2}, we have
\begin{equation}\label{extra term far from origin}
\color{black}
\begin{split}
&\int_{\delta<|h|<2\delta}\|\Delta_{2h}f\|^{q}_{L_{p}([a-2\delta,b+2\delta])}\frac{\mathrm{d}h}{|h|^{n}}\\
&\lesssim  \int_{\delta<|h|<2\delta}\left(\|f(\cdot+2h)\|^{q}_{L_{p}([a-2\delta,b+2\delta])}+\|f\|^{q}_{L_{p}([a-2\delta,b+2\delta])}\right)\frac{\mathrm{d}h}{|h|^{n}}\\
&\lesssim \|f\|^{q}_{L_{p}([a-4\delta,b+4\delta])} \int_{\delta<|h|<2\delta}  \frac{\mathrm{d}h}{|h|^{n}}\\
&\lesssim (N+4)^{nq}\|f\|^{q}_{L_{p}(\T^{n})}\int_{\delta<|h|<2\delta}\frac{\mathrm{d}h}{|h|^{n}}.
\end{split}
\end{equation}
Then \eqref{midd contr 2}, \eqref{esti S31} and \eqref{extra term far from origin} give the {\color{black} desired} control of $S_{3}$:
\begin{equation*}\label{esti S3}
\begin{split}
S_{3}& {\color{black} \lesssim} \|\nabla\varrho\|_{L_{\infty}(\R^{n})}(N+2)^{n}\left(\int_{\T^{n}} |\eta|^{-q}  \|\Delta_{\eta}^{2}f\|^{q}_{L_{p}(\T^{n})} \frac{\mathrm{d}\eta}{|\eta|^{n}}\right)^{\frac{1}{q}}\\
& {\color{black} + (N+4)^{n}} \|\nabla\varrho\|_{L_{\infty}(\R^{n})}\|f(x)\|_{L_{p}(\T^{n})}\left(\int_{\delta<|h|<2\delta}\frac{\mathrm{d}h}{|h|^{n}}\right)^{\frac{1}{q}}.
\end{split}
\end{equation*}
Inserting  the above estimates for $S_{1}$, $S_{2}$ and $S_{3}$ into \eqref{0 order reduction non-int s}, we have
\begin{equation*}\label{contr non-int s norm}
\begin{split}
\|f\varrho\|_{B^{s}_{pq}(\mathbb{R}^{n})}&
{\color{black} \lesssim} \|f\|_{L_{p}(\T^{n})}+\left(\int_{\T^{n}}|h|^{-q}\|\Delta^{2}_{h}f\|^{q}_{L_{p}(\T^{n})}\frac{\mathrm{d}h}{|h|^{n}}\right)^{\frac{1}{q}},
\end{split}
\end{equation*}
{\color{black} where the estimate depends only on \(\varrho\), \(n\) and \(\delta\) in the way detailed in the above inequalities}. This concludes the proof for the case when $s$ is an integer.
\end{proof}

{\color{black} Using Lemma~\ref{per nonper control} we shall now prove the following result.}

\begin{theorem}[{\color{black} A} composition theorem for  Besov spaces on $\T$]\label{periodic composition theorem} 
Let $1<p<\infty$, $1< q\leq \infty$ and $s>1+\frac{1}{p}$. For a Borel measurable function $f:\mathbb{R}\to\R$, the composition operator $T_{f}$ acts on $B^{s}_{pq}(\T)$ if and only if $f\in B^{s}_{pq,\mathrm{loc}}(\R)$. Moreover, $T_{f}$ is bounded {\color{black} if} $T_{f}$ acts on $B^{s}_{pq}(\T)$. If $f(0)=0$ holds additionally, we have
\begin{equation}\label{composition bound periodic}
\|{\color{black} f \circ g}\|_{B^{s}_{pq}(\T)}\lesssim \|(f {\color{black} \varphi_a})'\|_{B^{s-1}_{pq}(\R)}
\|g\|_{B^{s}_{pq}(\T)}(1+\|g\|_{B^{s}_{pq}(\T)})^{s-1-\frac{1}{p}},
\end{equation}
{\color{black} where \(a = \|g\|_{L_{\infty}}\) and \(\varphi_a\) is the \(a^{-1}\)-dilation of $\phi$ defined in Remark \ref{norm control for compsosition}.}
\end{theorem}

\begin{proof}[Proof of Theorem \ref{periodic composition theorem} ]
As before, we will only give full detail for the case $1\leq q<\infty$, since the proof for $q=\infty$ can be treated {\color{black} using the same procedure with only minor changes, and is in fact easier.} In addition, since $f(x)=(f(x)-f(0))+f(0)$ and $B^{s}_{pq}\hookrightarrow L_{\infty}$ for $s>\frac{1}{p} $, it suffices to prove the theorem for $f$ with $f(0)=0$.\\

\emph{\color{black} We first prove the sufficiency part of the result, namely that \(T_f  \colon  B^s_{pq}(\T) \to B^s_{pq}(\T)\) if \(f \in B^s_{pq,\mathrm{loc}}(\R)\).} {\color{black} Let \(g \in B^s_{pq}(\T)\),} and choose $\varrho\in C_{c}^{\infty}(\R)$ such that $\varrho=1$ on $(-4\pi,4\pi)$ with $\color{black} \mathrm{supp}(\varrho) \subset [-6\pi,6\pi]$. Denoting by $(g\varrho)_{\mathrm{per}}$ the periodic extension of $g\varrho$ from $[-\pi,\pi]$ to $\R$, we notice that the $2\pi$-periodic function $(g\varrho)_{\mathrm{per}}$ on $\R$ can be identified with $g$ on $\T$. By  \eqref{Besov norms on torus} and \eqref{periodic torus norm identify} we {\color{black} then} have
\begin{equation*}\label{key relation peri-nonperi}
\begin{split}
\|{\color{black} f \circ g} \|_{B^{s}_{pq}(\mathbb{T})}
&=\|{\color{black} f \circ (g\varrho)_{\mathrm{per}}} \|_{W^{[s]^{-}}_{p}([-\pi,\pi])}\\
&\quad+\sum_{|\alpha|=[s]^{-}}\bigg(\int_{[-\pi,\pi]}|h|^{-\{s\}q}\|\Delta_{h}^{2}\partial^{\alpha}[f\circ(g\varrho)_{\mathrm{per}}]\|^{q}_{L_{p}([-\pi,\pi])}\frac{\mathrm{d}h}{|h|}\bigg)^{\frac{1}{q}}\\
&\leq \|{\color{black} f \circ (g\varrho)}\|_{W^{[s]^{-}}_{p}(\R)}+\sum_{|\alpha|=[s]^{-}}\bigg(\int_{\R}|h|^{-\{s\}q}\|\Delta_{h}^{2}\partial^{\alpha}f(g\varrho)\|^{q}_{L_{p}(\R)}\frac{\mathrm{d}h}{|h|}\bigg)^{\frac{1}{q}}
\end{split}
\end{equation*}
Then Theorem \ref{composition theorem}, Remark \ref{norm control for compsosition} and Lemma \ref{per nonper control}  imply that
\begin{equation*}
\begin{split}
\|{\color{black} f \circ g}\|_{B^{s}_{pq}(\mathbb{T})}&\leq \|{\color{black} f \circ (g\varrho)}\|_{B^{s}_{pq}(\R)}\\
&\lesssim\|(f\varphi_{\|g\varrho\|_{L_{\infty}}})'\|_{B^{s-1}_{pq}(\R)}\|g\varrho\|_{B^{s}_{pq}(\R)}(1+\|g\varrho\|_{B^{s}_{pq}(\R)})^{s-1-\frac{1}{p}}\\
&\lesssim \|(f\varphi_{\|g\|_{L_{\infty}}})'\|_{B^{s-1}_{pq}(\R)}\|g\|_{B^{s}_{pq}(\mathbb{T})}(1+\|g\|_{B^{s}_{pq}(\T)})^{s-1-\frac{1}{p}},
\end{split}
\end{equation*}
{\color{black} with $\varphi_{(\cdot)}$ as in} Remark \ref{norm control for compsosition}.\\

{\color{black}  \emph{We now prove the necessity part of the result, showing that \(f \in B^s_{p,q,\text{loc}} (\R)\) if \(T_f\) acts on \(B^{s}_{p,q}(\T)\).} For any $\varrho\in C^{\infty}_{c}(\R)$ with compact support $\mathrm{supp}(\varrho) = [a,b]$, divide the interval \(I = [a-1,b+1]\) into a finite number of subintervals 
\[
I_j = [a -1 + (j-1)\eta, a -1 + j\eta], \qquad j = 1, 2, \ldots, (b-a +2)/\eta,
\]
of fixed but small length \(\eta \ll 1\) such that \((b-a + 2)/\eta\) is an integer. Then define smooth and compactly supported functions \(g_j\), \(j = 1, \ldots, (b-a +2)/\eta\), with \(|\supp(g_j)| \leq 1\) and
\[
g_j(x) = x, \qquad x \in I_j.
\]
Because the support of \(g_j\) has length less than \(2\pi\), we may extend $g_{j}$ \(2\pi\)-periodically to $\R$; denote this periodic extension by $g_{j,\mathrm{per}}$. We furthermore identify $g_{j,\mathrm{per}}$ as a function on the torus $\T$, with support of at most unit length.}\\

By {\color{black} definition of the Besov spaces and by} Remark~\ref{index s and h}, we  have
\begin{equation}\label{nece difce-difal part}
\begin{split}
\|f\varrho\|_{B^{s}_{pq}(\R)}
&\leq \|\varrho\|_{C^{[s]^{-}}(\R)}\|f\|_{W^{[s]^{-}}_{p}([a,b])}\\
&\quad+\sum_{|\alpha|\leq [s]^{-}}\bigg(\int_{|h|\leq \delta}|h|^{-\{s\}^{+}q}\|\Delta^{2}_{h}\partial^{\alpha}(f\varrho)\|_{L_{p}([a-2\delta,b+2\delta])}^{q}\frac{\diff h}{|h|}\bigg)^{\frac{1}{q}},
\end{split}
\end{equation}
{\color{black}  where we consider \(\delta \ll 1\) fixed but small enough for \(4\delta \leq 1\) (for larger \(\delta\) one needs to adjust the intervals \(I_j\) above).}
The estimate for $\|f\|_{W^{[s]^{-}}_{p}([a,b])}$ is straight-forward since 
\begin{equation}\label{local estimate comp}
{\color{black}
\begin{aligned}
\|f\|_{W^{[s]^{-}}_{p}([a,b])}&\leq \|f\|_{W^{[s]^{-}}_{p}(I)}\\
&\leq \sum_{j} \|f \circ g_{j}\|_{W^{[s]^{-}}_{p}(I_j)}\\ 
&\leq \sum_{j} \|f \circ ({g}_{j,\mathrm{per}})\|_{W^{[s]^{-}}_{p}(\T)}.
\end{aligned}
}
\end{equation}
{\color{black}  Note that the last term in \eqref{local estimate comp} is finite since, by assumption, $T_{f}$ acts on $B^{s}_{pq}(\T)$. 
The estimate of the second term in \eqref{nece difce-difal part} however needs a more careful analysis, similar to the one in the proof of Lemma \ref{per nonper control}, but it suffices to consider the case when $\alpha=0$ since the case when $\alpha\in [1,[s]^{-}]$ may be reduced to the former case by the Leibniz rule.} We are then left to control the term
\begin{equation}
\mathrm{S}(f,\varrho) =\bigg(\int_{|h|\leq \delta}|h|^{-\{s\}^{+}q}\|\Delta^{2}_{h}(f\varrho)\|_{L_{p}([a-2\delta,b+2\delta])}^{q}\frac{\diff h}{|h|}\bigg)^{\frac{1}{q}}
\end{equation}
{\color{black} from \eqref{nece difce-difal part}.}
The argument for estimating $\mathrm{S}(f,\varrho)$ varies according to whether $s$ is an integer or not. Therefore, we consider the two cases seperately.

\emph{The case when $s$ is not an integer}. In this case $\{s\}^{+}\in (0,1)$, and by \eqref{2 order iteration} we have
\begin{equation*}\label{control s nonint comp}
\begin{aligned}
\mathrm{S}(f,\varrho)&\leq \bigg(\int_{|h|\leq \delta}|h|^{-\{s\}^{+}q}\|(\Delta^{2}_{h}f)\varrho(\cdot+h)\|_{L_{p}([a-2\delta,b+2\delta])}^{q}\frac{\diff h}{|h|}\bigg)^{\frac{1}{q}}\\
&\quad+ \bigg(\int_{|h|\leq \delta}|h|^{-\{s\}^{+}q}\|f(\cdot+2h) \Delta^{2}_{h}\varrho\|_{L_{p}([a-2\delta,b+2\delta])}^{q}\frac{\diff h}{|h|}\bigg)^{\frac{1}{q}}\\
&\quad+ \bigg(\int_{|h|\leq \delta}|h|^{-\{s\}^{+}q}\|(\Delta_{2h}f) \Delta_{h}\varrho\|_{L_{p}([a-2\delta,b+2\delta])}^{q}\frac{\diff h}{|h|}\bigg)^{\frac{1}{q}}\\
&=:T_{1}+T_{2}+T_{3}.
\end{aligned}
\end{equation*}
We first estimate $T_{1}$:
\begin{equation*}\label{control s nonint comp T1}
{\color{black}
\begin{split}
{\textstyle \frac{1}{2}} T_{1} &=  \|\varrho\|_{L_{\infty}(\R)}\bigg(\sum_{j} \int_{|h|\leq \delta}|h|^{-\{s\}^{+}q}\|\Delta^{2}_{h}(f \circ g_{j})\|_{L_{p}(I_j \cap [a-2\delta, b + 2\delta])}^{q}\frac{\diff h}{|h|}\bigg)^{\frac{1}{q}}\\
&\leq \|\varrho\|_{L_{\infty}(\R)}\bigg(\sum_{j} \int_{|h|\leq \delta}|h|^{-\{s\}^{+}q}\|\Delta^{2}_{h}(f \circ {g}_{j})\|_{L_{p}(\R)}^{q}\frac{\diff h}{|h|}\bigg)^{\frac{1}{q}}\\
&\leq \|\varrho\|_{L_{\infty}(\R)}\sum_{j} \bigg(\int_{\T}|h|^{-\{s\}^{+}q}\|\Delta^{2}_{h}(f \circ {g}_{j,\mathrm{per}})\|_{L_{p}(\T)}^{q}\frac{\diff h}{|h|}\bigg)^{\frac{1}{q}}.
\end{split}
}
\end{equation*}
{\color{black} By the mean value theorem and the variable substitution $x+2h\mapsto y$ we may then estimate $T_{2}$ similarly as follows.
\begin{equation*}\label{control s nonint comp T2}
\begin{split}
{\textstyle \frac{1}{2}} T_{2}
&= \|\varrho''\|_{L_{\infty}(\R)}\bigg(\int_{|h|\leq \delta}|h|^{-\{s\}^{+}q}|h|^{2q}\|f\|_{L_{p}([a-2\delta+2h,b+2\delta+2h])}^{q}\frac{\diff h}{|h|}\bigg)^{\frac{1}{q}}\\
&\leq \|\varrho''\|_{L_{\infty}(\R)}\bigg(\sum_{j} \int_{|h|\leq \delta}|h|^{(2-\{s\}^{+})q}\|f \circ g_{j}\|_{L_{p}(I_j)}^{q}\frac{\diff h}{|h|}\bigg)^{\frac{1}{q}}\\
&\leq \|\varrho''\|_{L_{\infty}(\R)}\bigg(\int_{|h|\leq \delta}|h|^{(2-\{s\}^{+})q}\frac{\diff h}{|h|}\bigg)^{\frac{1}{q}}\sum_{j} \|f \circ {g}_{j} \|_{L_{p}(\R)}\\
&\leq \|\varrho''\|_{L_{\infty}(\R)}\bigg(\int_{|h|\leq \delta}|h|^{(2-\{s\}^{+})q}\frac{\diff h}{|h|}\bigg)^{\frac{1}{q}}\sum_{j} \|f \circ {g}_{j,\mathrm{per}} \|_{L_{p}(\T)}.
\end{split}
\end{equation*}
The term $T_{3}$ can be estimated similarly as $T_{2}$.
\begin{equation*}\label{control s nonint comp T3}
\begin{split}
T_{3}
 &\lesssim \|\varrho'\|_{L_{\infty}(\R)}\bigg(\int_{|h|\leq \delta}|h|^{(1-\{s\}^{+})q}\|\Delta_{2h}f\|_{L_{p}([a-2\delta,b+2\delta])}^{q}\frac{\diff h}{|h|}\bigg)^{\frac{1}{q}}\\
&\lesssim \|\varrho'\|_{L_{\infty}(\R)}\bigg(\int_{|h|\leq \delta}|h|^{(1-\{s\}^{+})q}\|f\|_{L_{p}(I)}^{q}\frac{\diff h}{|h|}\bigg)^{\frac{1}{q}}\\
&\lesssim \|\varrho'\|_{L_{\infty}(\R)}\bigg(\int_{|h|\leq \delta}|h|^{(1-\{s\}^{+})q}\frac{\diff h}{|h|}\bigg)^{\frac{1}{q}}\sum_{j} \|f \circ {g}_{j,\mathrm{per}}\|_{L_{p}(\T)}.
\end{split}
\end{equation*}
By the assumption that $f$ acts on $B^{s}_{pq}(\T)$ it is clear from the above estimates that $S(f,\varrho)$, and thus also \(\|f \varrho\|_{B^s_{pq}(\R)}\),  is bounded for each \(\varrho \in C_c^\infty\).\\

\emph{The case when $s$ is an integer}. In this case $\{s\}^{+}=1$, and we first notice that
\begin{equation*}\label{control s int comp}
\begin{split}
\mathrm{S}(f,\varrho)&\leq \bigg(\int_{|h|\leq \delta}|h|^{-q}\|(\Delta^{2}_{h}f) \varrho(\cdot+h)\|_{L_{p}([a-2\delta,b+2\delta])}^{q}\frac{\diff h}{|h|}\bigg)^{\frac{1}{q}}\\
&\quad+ \bigg(\int_{|h|\leq \delta}|h|^{-q}\|f(\cdot +2h) \Delta^{2}_{h}\varrho\|_{L_{p}([a-2\delta,b+2\delta])}^{q}\frac{\diff h}{|h|}\bigg)^{\frac{1}{q}}\\
&\quad+ \bigg(\int_{|h|\leq \delta}|h|^{-q}\|(\Delta_{2h}f) \Delta_{h}\varrho \|_{L_{p}([a-2\delta,b+2\delta])}^{q}\frac{\diff h}{|h|}\bigg)^{\frac{1}{q}}\\
&=:S_{1}+S_{2}+S_{3}.
\end{split}
\end{equation*}
We estimate $S_{1}$ and $S_{2}$ analogously to $T_{1}$ and $T_{2}$, by
\begin{equation*}\label{control s int comp S1 S2}
\begin{split}
S_{1}&\lesssim \|\varrho\|_{L_{\infty}(\R)}\sum_{j} \bigg(\int_{\T}|h|^{-q}\| \Delta^{2}_{h}(f \circ {g}_{j,\mathrm{per}}) \|_{L_{p}(\T)}^{q}\frac{\diff h}{|h|}\bigg)^{\frac{1}{q}},\\
S_{2}&\lesssim \|\varrho''\|_{L_{\infty}(\R)}\bigg(\int_{|h|\leq \delta}|h|^{q}\frac{\diff h}{|h|}\bigg)^{\frac{1}{q}}\sum_{j} \| f \circ {g}_{j,\mathrm{per}} \|_{L_{p}(\T)}.
\end{split}
\end{equation*}
The estimate for $S_{3}$, on the other hand, starts in a way identical to \eqref{0 order reduction non-int s}--\eqref{midd contr 2}. Just as in~\eqref{midd contr 2} that yields
\begin{equation}\label{midd contr 2 comp}
\begin{split}
S_{3}&\leq \|\nabla\varphi\|_{L_{\infty}(\R)}\left(S_{31}+\left(\int_{\delta<|h|<2\delta}\|\Delta_{2h}f\|^{q}_{L_{p}([a-2\delta,b+2\delta])}\frac{\mathrm{d}h}{|h|}\right)^{\frac{1}{q}}\right),
\end{split}
\end{equation}
with
\begin{equation}\label{esti S31 comp}
\begin{split}
S_{31} &= \left(\int_{|h| < \delta} \|\Delta_{2h}^2 f\|^{q}_{L_{p}([a-2\delta,b+2\delta])}\frac{\mathrm{d}h}{|h|}\right)^{\frac{1}{q}}\\
&\leq\sum_{j} \left( \int_{|\eta|<2\delta}|\eta|^{-q}\|\Delta_{\eta}^{2} (f \circ g_j)\|^{q}_{L_{p}(I_j)}\frac{\mathrm{d}\eta}{|\eta|}\right)^{\frac{1}{q}}\\
&\leq\sum_{j} \left( \int_{|\eta|<2\delta}|\eta|^{-q}\|\Delta_{\eta}^{2} (f \circ g_j)\|^{q}_{L_{p}(\R)}\frac{\mathrm{d}\eta}{|\eta|}\right)^{\frac{1}{q}}\\
&\leq \sum_{j} \left( \int_{\T}|\eta|^{-q}\|\Delta_{\eta}^{2} (f \circ {g}_{j,\mathrm{per}})]\|^{q}_{L_{p}(\T)}\frac{\mathrm{d}\eta}{|\eta|}\right)^{\frac{1}{q}},
\end{split}
\end{equation}
by the construction of the functions \(\{g_j\}_j\). For the remaining term in \eqref{midd contr 2 comp}, it may similarly be estimated as
\begin{equation}\label{extra term far from origin comp}
\begin{split}
\int_{\delta<|h|<2\delta}\|\Delta_{2h}f\|^{q}_{L_{p}([a-2\delta,b+2\delta])}\frac{\mathrm{d}h}{|h|} & \leq 2\sum_{j}  \|f \circ {g}_{j,\mathrm{per}}\|_{L_{p}(\T)}\int_{\delta<|h|<2\delta}\frac{\mathrm{d}h}{|h|}.
\end{split}
\end{equation}
Then \eqref{midd contr 2 comp}, \eqref{esti S31 comp} and \eqref{extra term far from origin comp} together give control of $S_{3}$:
\begin{equation}\label{control s int comp S3}
\begin{split}
S_{3}&\leq \|\varrho'\|_{L_{\infty}(\R)}\sum_{j} \Bigg( \left(\int_{\T}|h|^{-q}\|\Delta_{h}^{2} (f \circ {g}_{j,\mathrm{per}}) \|^{q}_{L_{p}(\T)}\frac{\mathrm{d}h}{|h|}\right)^{\frac{1}{q}}\\
&\quad+2 \|f \circ {g}_{j,\mathrm{per}}\|_{L_{p}(\T)}\int_{\delta<|h|<2\delta}\frac{\mathrm{d}h}{|h|}\Bigg)\\
\end{split}
\end{equation}
Since we are still proving the necessity part of the theorem, $f \circ {g}_{j,\mathrm{per}} \in B^{s}_{pq}(\T)$ by assumption. It is then clear from the above estimates that $S(f,\varrho)$, and thus \(\|f\|_{B^s_{pq,\text{loc}}(\R)}\), is bounded.}
\end{proof}

{\color{black}
For $s>0$ and $p=2$, it is well known that fractional Sobolev spaces are equivalent with the Bessel--Potential spaces, $W^{s}_{2}(X)=H^{s}(X)$ with \(X \in \{\R^n, \T^n\}\) in our case. A consequence of this is that the proof of Theorem~\ref{main theorem} can be followed in detail with the standard Sobolev spaces replaced by their periodic counterparts and the Fourier transform replaced accordingly, as described above. Note here that the embedding \(H^{s} \hookrightarrow BC^{1}\) for \(s > 3/2\) is equally valid in the periodic case, and Kato--Ponce commutator estimates in the periodic setting are available in \cite{MR2333212}. Analogous to the case on $\mathbb{R}$, one then obtains Theorem~\ref{main theorem} for periodic initial data.

\section*{Acknowledgements}
L. P.  wants to thank W. Sickel and H. Triebel for their guidance and helpful comments that helped in the construction of the composition theorem in the periodic setting. 
}

%

\end{document}